\documentclass[11pt]{amsart}

\usepackage[subpreambles=false]{standalone}
\usepackage{import}
\usepackage{multicol}
\usepackage[english]{babel}
\usepackage{tikz}
\usetikzlibrary{backgrounds}

\usepackage{url}
\usepackage{hyperref}
\usetikzlibrary{calc, babel, cd}
\usepackage{xcolor}
\usepackage{float}
\usepackage{amsfonts}%
\usepackage{amssymb,amsmath,amstext,appendix, amsthm, mathtools}
\usepackage{mathrsfs}

\usepackage[colorinlistoftodos]{todonotes}

\newtheorem{theorem}{Theorem}[section]
\newtheorem{lemma}[theorem]{Lemma}

\newtheorem{corollary}[theorem]{Corollary}
\newtheorem{question}[theorem]{Question}
\theoremstyle{definition}
\newtheorem{definition}[theorem]{Definition}
\newtheorem{example}[theorem]{Example}
\newtheorem{remark}[theorem]{Remark}

\newcommand\R{\mathbb R}

\newcommand\A{{\overline{A}}}

\newcommand\CA{{\mathcal A}}
\newcommand\CB{{\mathcal B}}
\newcommand\CC{{\mathcal C}}

\newcommand\CP{{\mathcal P}}
\newcommand\CR{{\mathcal R}}

\DeclareMathOperator\conv{conv}

\title[Associahedra minimize $f$-vectors of secondary polytopes]{Associahedra minimize $f$-vectors of secondary polytopes of planar point sets}
\author{Antonio Fern\'andez and Francisco Santos}
\date{}
\thanks{Supported by grants PID2022-137283NB-C21 funded by MCIN/AEI/10.13039/501100011033, by FPU19/04163 of the Spanish Government and by project CLaPPo (21.SI03.64658) of Universidad de Cantabria and Banco Santander}

\address{
Departamento de Matem\'aticas, Estad\'istica y Computaci\'on\\
Universidad de Cantabria\\
39005 Santander, Spain
}
\email{francisco.santos@unican.es, antonio.fernandezg97@gmail.com}

\begin{document}

\begin{abstract}
Kupavskii,  Volostnov, and Yarovikov have recently shown that any set of $n$ points in general position in the plane has at least as many (partial) triangulations as the convex $n$-gon. We generalize this in two directions: we show that \emph{regular} triangulations are enough, and we extend the result to all regular subdivisions, graded by the dimension of their corresponding face in the secondary polytope.
\end{abstract}

\maketitle

\setcounter{tocdepth}{1}

\section{Introduction}
Triangulating a finite set of points $\CA\subset \R^2$ is a very useful tool both from a theoretical and an applied point of view. In particular, the problem of how many triangulations a point set can have has attracted some attention. 

The most classical example is that of points in \emph{convex position}; that is, when $\CA$ is the vertex set of a convex $n$-gon. Then, the exact number of triangulations equals the Catalan number
\[
C_{n-2} := \frac{1}{n+1}\binom{2n-4}{n-2} \in \Theta(4^n n^{-3/2}),
\]
a formula already known to Euler. See, e.g., \cite[Theorem 1.1.2]{triangbook} for a proof.

Configurations with more and with fewer triangulations than the convex $n$-gon are known, but before we delve into discussing the topic it is convenient to make two precisions:
\begin{enumerate}
\item We distinguish between \emph{partial} and \emph{full} triangulations. In both cases a triangulation of $\CA$ is a subdivision of $\conv(\CA)$ into triangles using only points from $\CA$ as vertices and intersecting edge-to-edge (that is, the intersection of two triangles is either empty, a vertex, or a full edge). But in \emph{full triangulations} we require that all the points of $\CA$ are used as vertices, while in \emph{partial triangulations} we allow for some of them not to be used. (For points in convex position all triangulations are full).

\item We will only consider point sets \emph{in general position}, that is, with no three of them on the same straight line. Without this restriction, the configuration with the least number of triangulations consists of the vertices of a triangle plus $n-3$  points along one edge of it; this has a single full triangulation, and $2^{n-3}$ partial triangulations. 

Any configuration can be slightly perturbed into general position, and  perturbation can only increase the number of triangulations.
\end{enumerate}

It is  known that there exist global constants $c_1 < c_2$ such that every configuration of $n$ points in general position has more than $c_1^n$ and fewer than $c_2^n$  triangulations, be them partial or full. (In fact, proving this for full triangulations is enough, since partial triangulations of $\CA$ are full triangulations of subconfigurations of $\CA$).
The configurations with the largest and smallest numbers of full triangulations known are the so-called ``Koch chain'' and ``double circle". They have respectively
\[
\Omega^*(9.08^n),
\qquad\text{and}\qquad
\Theta^*((2\sqrt{3})^n) = \Theta^*(3.464^n)
\]
full triangulations, where the notations $\Omega^*(\ )$ and $\Theta^*(\ )$ indicate that a polynomial factor is neglected. See~\cite[Section 3.3]{triangbook} for more information and \cite{lower,upper,RW23} for the original constructions. 

The starting point of this paper is the recent proof by Kupavskii,  Volostnov, and Yarovikov \cite{KVY21} that the configuration  minimizing the number of \emph{partial} triangulations for a given number of points is precisely the convex $n$-gon. Our first observation is that their method of constructing at least Catalan many triangulations for an arbitrary configuration actually produces triangulations that are \emph{regular}. That is:

\begin{theorem}
\label{thm:main}
Any configuration $\CA$ of $n$ points in general position in the plane has at least $C_{n-2}$ regular (partial) triangulations.
\end{theorem}

Observe that the same is not true for full triangulations (the double circle is a counter-example) or for partial triangulations of point sets in non-general position ($n-1$ collinear points plus an extra one is a counter-example).

For the  definition and properties of regular triangulations and subdivisions see Section~\ref{sec:prelim}; one fundamental result regarding them is:

\begin{theorem}[Gelfand, Kapranov and Zelevinsky~\cite{GKZpaper,GKZbook}, see also~\protect{\cite[Thms. 5.1.9 and 5.2.16]{triangbook}}]
\label{thm:secondary}
The decomposition of the space $\R^n$ of height vectors $\omega$ according to what regular subdivision of $\CA$ they produce is a complete polyhedral fan and it is the normal fan of a polytope $\Sigma(\CA)$ of dimension $n-3$, called the \emph{secondary polytope of $\CA$}.
\end{theorem}

\begin{example}[The associahedron]
If $\CA$ consists of $n$ points in convex position then every subdivision of $\CA$ is regular, so the face poset of the secondary polytope equals the poset of all polyhedral subdivisions of $\CA$. 
In turn, polyhedral subdivisions of $\CA$ are in bijection to noncrossing sets of diagonals, so that they form a simplicial complex of dimension $n-4$, independent of which particular set of points (in convex position) we started with. Theorem~\ref{thm:secondary} tells us that this simplicial complex is dual to a simple $(n-3)$-polytope, called the $(n-3)$-\emph{associahedron}~\cite{Lee89} (see also \cite[Sect. 1.1]{triangbook} or \cite[pp.18, 306]{Ziegler}). Subdivisions using $k$ diagonals of the $n$-gon correspond to faces of dimension $n-3-k$ in the associahedron.

The Catalan formula for triangulations of the $n$-gon was generalized by Cayley to the following formula counting the number of sets of $n-3-k$ non-crossing diagonals, that is, the number of faces of dimension $k$ of the associahedron; for a proof see, e.g.,~\cite[Theorem 3]{Lee89}, observing that the $j$ there equals  $n-3-k$  here:
\begin{align}
C_{n-2}^k = \frac1{n-1} \binom{n-3}{k} \binom{2n-4-k}{n-2}.
\label{eq:f-vector}
\end{align}

For example, $C_{n-2}^0=C_{n-2}$, corresponding to triangulations, and $C_{n-2}^{n-3}=1$, corresponding to the trivial subdivision. Similarly,
$C_{n-2}^{n-4}=n(n-3)/2$, corresponding to the number of diagonals of the $n$-gon.

The numbers $C_n^k$ form a triangle that appears as sequence A033282 in the Online Encyclopedia of Integer Sequences~\cite{OEIS} and are related to the better known Narayana numbers, which give (among other combinatorial interpretations) the $h$-vector of the associahedron.
\end{example}

Theorem~\ref{thm:secondary} suggests that we can call ``dimension'' of a regular subdivision its dimension as a face of the secondary polytope, and then compare the number of subdivisions of each dimension for a configuration $\CA$ to the same  number for a convex $n$-gon. Our main result is that the ideas in the proof of Theorem~\ref{thm:main} can be adapted to this more general case, and  again the convex $n$-gon minimizes the number of regular subdivisions of each dimension:

\begin{theorem}[Main theorem]
\label{thm:main2}
For every point configuration $\CA$ of size $n$ in general position in the plane, and for every $k\in\{0,\dots,n-3\}$, the secondary polytope $\Sigma(\CA)$ has at least as many faces of dimension $k$ as the secondary polytope of the $n$-gon (the $(n-3)$-associahedron), given by formula \eqref{eq:f-vector}.
\end{theorem}


In the next section we review some concepts and properties from the theory of regular subdivisions. Then, in  Section~\ref{sec:triangs} we prove Theorem~\ref{thm:main}. The proof is essentially that from \cite{KVY21} except that we prove regularity, and that along the way we introduce a formalism aimed at proving the more general Theorem~\ref{thm:main2}. The latter is proved in Section~\ref{sec:subdivs}.



%

One can ask whether there is an analogue of Theorem~\ref{thm:main2} in higher dimension. We include some remarks regarding this question in Section~\ref{sec:higherdim}. For example, we show how the problem, restricted to configurations with four more points than their dimension, is very much connected to the \emph{geodesic crossing number} of the complete graph in the $2$-sphere.

\section{Preliminaries on regular subdivisions}
\label{sec:prelim}
Although our paper (except for the remarks in Section~\ref{sec:higherdim}) deals only with dimension two, let us here recall the formalism of subdivisions and secondary polytopes in arbitrary dimension. A comprehensive reference on the topic is \cite{triangbook}.

Let $\CA=\{P_1,\dots,P_n\}\subset \R^d$ be a point configuration of size $n$, that is, a set of $n$ points from $\R^n$. A polyhedral subdivision $T$ of $\CA$ is, loosely speaking, a decomposition of the convex hull $\conv(\CA)$ into convex polytopes that intersect properly, meaning that the intersection of any two of them is a face of both, and with the property that all vertices of all these polytopes are taken from $\CA$. The individual polytopes of the decomposition are called the \emph{cells} of $T$. 
However, for reasons that will become apparent later, it is convenient to have a more combinatorial definition where a subdivision is not a collection of geometric cells (polytopes) but of ``combinatorial cells'' (subsets of $\CA$):

\begin{definition}[Polyhedral subdivision~\protect{\cite[Theorem 4.55]{triangbook}}]
Let $\CA\subset \R^d$ be a point configuration.
A \emph{(polyhedral) subdivision} of $\CA$ is a collection $T=\{\CC_1,\dots,\CC_k\}$ of subsets of $\CA$, called \emph{cells of $T$} with the following properties:
\begin{enumerate}
    \item Each $\CC_i$ is full dimensional. That is, $\dim(\conv(\CC_i))= d$.
    \item The $\CC_i$ cover $\conv (\CA)$. That is,
    \[\cup_i \conv(\CC_i) = \conv(\CA).\]
    \item The $\CC_i$ intersect properly. That is, for every $i,j$ we have that  $F:=\conv(\CC_i)\cap \conv(\CC_j)$ is a common face of $\conv(\CC_i)$ and $\conv(\CC_j)$ and, moreover,
    \[
    F \cap \CC_i = F \cap \CC_j.
    \]
\end{enumerate}
\end{definition}

We do not require that all the points of $\CA$ are vertices of a cell. So, what we call subdivisions from now on would correspond to the ``partial subdivisions'' in the introduction. Also, observe that a cell $\CC_i\subset \CA$ may contain elements of $\CA$ that are not vertices of $\conv(\CC_i)$. In particular, different ``combinatorial'' subdivisions may correspond to the same ``geometric'' subdivision where by geometric subdivision we mean the collection of subpolytopes $\conv(\CC_i)$.

 Subdivisions of $\CA$ form a poset under refinement, where $T$ refines $T'$ if each cell of $T$ is contained in some cell of $T'$. The unique maximal (i.e., least refined) element in this poset is the trivial subdivision $\{\CA\}$ and the minimal elements are the (partial)  \emph{triangulations} of $\CA$: the subdivisions in which all cells are affinely independent.

A subdivision $T$ of $\CA=\{P_1,\dots,P_n\}$
is \emph{regular} if it can be obtained from a \emph{lifting vector} or \emph{height vector} $\omega=(\omega_1,\dots,\omega_n)\in \R^n$ as follows: Consider the lifted point configuration
\[
\tilde\CA:=\{(P_i,\omega_i): i=1,\dots,n\}\subset \R^{d+1}
\]
and take as cells of the \emph{regular subdivision of $\CA$ induced by $\omega$} the projections in $\CA$ of the lower facets of $\tilde\CA$. Here, a facet of $\tilde\CA$ is called \emph{lower} if the hyperplane containing that facet is not vertical and lies below $\conv(\tilde\CA)$. Observe that this definition makes some points of $\CA$ not part of any cell (those that are not in the lower hull of $\conv(\tilde\CA)$) but it also makes some points to be part of a cell and not vertices of it (those that are in the lower hull of $\conv(\tilde\CA)$ but are not vertices of $\conv(\tilde\CA)$).


Theorem~\ref{thm:secondary} implies that the poset of \emph{regular} subdivisions of $\CA$ under refinement is isomorphic to the lattice of (non-empty) faces of the secondary polytope of $\CA$, of dimension $n-3$. In particular, we have a well-defined \emph{dimension} associated to each regular subdivision $T$: the dimension of the face of $\Sigma(\CA)$ corresponding to $T$.
Regular triangulations have dimension zero (they biject to vertices of $\Sigma(\CA)$) and the trivial subdivision has dimension $n-3$ (it corresponds to $\Sigma(\CA)$ considered as a face of itself). 

To construct regular subdivisions in a controlled way, in this paper we use several times the following lemma:

\begin{lemma}[Regular refinement~\protect{\cite[Lemma 2.3.16]{triangbook}}]
\label{lemma:refinement}
Let $S$ be a regular subdivision of $\CA$, obtained for a certain height vector $\alpha\in \R^n$. Let $\omega\in \R^n$ be another height vector. Then, for any sufficiently small $\epsilon>0$, the regular subdivision of $\CA$ for the height vector $\alpha + \epsilon \omega$ equals the refinement of $S$ obtained subdividing each cell $\CC$ of $S$ as a regular subdivision given with the heights $\omega|_\CC$.
\end{lemma}

\section{The number of regular triangulations}
\label{sec:triangs}

Except for the regularity part (the proof of which we partially defer to the next section) in this section we are only rewriting the proof of the main result from \cite{KVY21}. Strictly speaking we could omit this section, since it is nothing but a special case of the result proved in the next one, but we believe  it is worth doing the case of triangulations first, since it is simpler and serves as a warm-up. Also, our way of writing this section is intended to serve as a preparation for the more general case in the next section.

Throughout the paper we let $\CA$ be a point configuration of size $n$ in general position in the plane. Let $A\in \CA$ be a vertex of $\conv (\CA)$ and let $B$ and $C$ be its adjacent vertices. We denote $P_0=B,P_1,\dots,P_{n-2}=C$  the $n-1$ points of $\CA\setminus \{A\}$, ordered counter-clockwise as seen from $A$. 

Any polygonal line from $B=P_0$ to $C=P_{n-2}$, and with vertex set an ordered (as seen from $A$) subset of $\CA\setminus \{A\}$ will be called an \emph{$A$-monotone polygonal line in $\CA$}, or a \emph{monotone polyline} for short. Since every subset of $\CA\setminus \{A\}$ containing $B$ and $C$ is the vertex set of a unique such polyline, there are exactly $2^{n-3}$ monotone polylines for vertex $A$. 

Following  \cite{KVY21}, to each monotone polyline $L$ we associate a signature $\sigma_L\in \{-1,1\}^{[n-3]}$ (that is, we define a sign for each point $P_i\in \CA\setminus \{A,B,C\}$), in the following way.
Write $L$ as a list of its vertices, namely $L=P_{a_0}P_{a_1}\cdots P_{a_{s-1}}P_{a_{s}}$, where $P_{a_0}=B$ and $P_{a_{s}}=C$. Given a point $P_i,1\leq i\leq n-3$, let $P_{a_l}$ and $P_{a_r}$ be the two points of $L$ with $a_l<i<a_r$ such that $r-l$ is minimal. Then, 
\begin{itemize}
    \item if $P_i$ is a point in $L$, we set $\sigma_L(P_i)=1$ if the segments $AP_i$ and $P_{a_l}P_{a_r}$ intersect, and $\sigma_L(P_i)=-1$ otherwise; 
    \item if $P_i$ is not in $L$, we set $\sigma_L(P_i)=-1$ if the segments $AP_i$ and $P_{a_l}P_{a_r}$ intersect, and $\sigma_L(P_i)=1$ otherwise. 
\end{itemize}
We say that the points of $\CA\setminus \{A\}$ with negative signature are \emph{below the polyline $L$} and those with positive signature are \emph{above}.

\begin{example}
In Figure~\ref{fig:step2} we see a configuration of  $22$ points $A,B, P_1,\dots$, $P_{19},B$ and a polyline in it, with vertices $B,P_3,P_6,P_9,P_{12},P_{13},P_{15}$, $P_{19},C$. The signature it produces is\\
\[
\small
\begin{array}{ccccccccccccccccccc}
1&2&3&4&5&6&7&8&9&10&11&12&13&14&15&16&17&18&19\\
\hline
-&+&-&-&-&+&-&-&-&-&-&-&+&-&-&-&-&+&-\\
\end{array}
\]
\end{example}

The following result is \cite[Lemma 2]{KVY21}:

\begin{lemma}
\label{lemma:signatures}
The map so obtained is a bijection between the $A$-monotone polygonal lines and the set $\{-1,+1\}^{[n-3]}$.
\end{lemma}

The following intuitive idea behind this construction can be considered an informal proof of the lemma: think of the points of $\CA$ as nails on a board and of $L$ as a rubberband that goes below some nails and above some others. Positive points are those that are either strictly above $L$ or where $L$ bends upwards (implying that the rubberband needs to go below the corresponding nails) and negative points are those below $L$ or where $L$ bends downwards.

\begin{example}
When $\CA$ is in convex position, the positive entries in $\sigma_L$ are the internal vertices of $L$ and negative entries are the points not in $L$.
\end{example}

We can apply this construction to triangulations of $\CA$. If $T$ is a triangulation, the link of $A$ in $T$ (that is, the sequence of segments that form a triangle with $A$) is an $A$-monotone polyline $L$. We set $\sigma_T=\sigma_L$ and call this sign vector the \emph{link signature of $T$ from $A$}.

We now prove Theorem~\ref{thm:main}.
Let $\CB=\{A, Q_0,Q_1,\dots$, $Q_{n-3}, Q_{n-2}\}$ be a point configuration in convex position and such that $P_i$ lies in the segment $AQ_i$ for every $i$. (That is, obtain each point $Q_i$ by moving each $P_i$ radially away from $A$ until it becomes a vertex of the configuration). We want to prove that
%
$\CA$ has at least as many regular triangulations as the convex $n$-gon $\CB$.
%
We show this by induction on the number of points of $\CA$,
and stratified by signature. That is, we prove the following statement, which trivially implies Theorem \ref{thm:main}:

\begin{lemma}
\label{lemma:main}
Let $\sigma \in \{-1,+1\}^{[n-3]}$ be a signature. Then, $\CA$ has at least as many regular triangulations with link signature equal to $\sigma$ as $\CB$.
\end{lemma}

In the proof we will need the following inequality between Catalan numbers.
The inequality follows from the definition of $C_n$, or from the fact that a convex $(k_1+\dots+k_m+2)$-gon can be subdivided into $m$ convex polygons of sizes $k_1+2, \dots, k_m+2$, and then these polygons can be triangulated independently:

\begin{lemma}[{\cite[Corollary 1]{KVY21}}]
\label{lemma:catalan}
For integers $k_1,\dots,k_m$, we have 
\[
C_{k_1}\cdots C_{k_m}\leq C_{k_1+\dots+k_m}.
\]
\end{lemma}

\begin{proof}[Proof of Lemma~\ref{lemma:main}  and Theorem~\ref{thm:main}]
Let $L=P_{a_0}P_{a_1}\cdots P_{a_{s-1}}P_{a_s}$ be the polyline associated to the signature $\sigma$ in $\CA$, with $B=P_{a_0}$ and $C=P_{a_s}$. We call a \emph{negative interval} of 
$\sigma$ (with respect to $\CA$)  each subset 
\[
\{i,j\}\cup ([i,j] \cap \sigma^-) ,
\]
where $a_i$ and $a_j$ are two consecutive points with non-negative signature along $L$. For the purpose of this definition, $B=P_{a_0}$ and $C=P_{a_s}=P_{n-2}$ are considered non-negative points in $\sigma$, so that the first negative interval starts at $0$ and the last one ends at $n-2$. We call \emph{length} of a negative interval its number of negative points, that is, the cardinality of $[i,j] \cap \sigma^-$. A negative interval of length $l$ has $l+2$ points,  the $l$ negative ones plus its  two end-points $i$ and $j$.

Observe that each negative interval corresponds to a maximal concave (as seen from $A$) chain in $L$. Here we call a chain $P_{a_i}\dots P_{a_j}$ in $L$ concave if for each intermediate point $P_{a_k}$ in the chain we have that $P_{a_k}$ is inside the triangle $AP_{a_{k-1}}P_{a_{k+1}}$ or, equivalently, if $L$ bends downwards at $P_{ak}$.
Being maximal implies that $P_{a_i}$ and $P_{a_j}$ are non-negative while the rest of the points in the chain are negative, and the negative interval corresponding to the chain is 
\[
\{i,j\}\cup ([i,j] \cap \sigma^-) .
\]

Let $\CP_1,\dots,\CP_l$ be the subconfigurations of $\CA$ consisting of the points in each negative interval. Also, for each segment $P_{a_{i-1}}P_{a_{i}}$ in $L$ ($i=1,\dots,s$)
let $\CR_i$ be the triangle $\{A,P_{a_{i-1}},P_{a_{i}}\}$.
The following claim is a particular case of Lemma~\ref{lemma:extended_star} that we prove in the next section. See Figure~\ref{fig:step2} for an illustration:

\medskip
\noindent{\bf Claim:}
There is a regular subdivision $S$ of $\CA$ containing as cells all the negative intervals $\CP_i$ and triangles $\CR_j$.

\medskip

Let $\alpha \in \R^n$ be a height function producing $S$ as a regular subdivision of $\CA$.

\begin{figure}[htb]
\centering
%

\begin{tikzpicture}[scale=0.5]
\fill (8,3) circle(7pt) coordinate (A) node[anchor=270]{$A$}
      (0,0) circle(4pt) coordinate (B) node[anchor=0]{$B$}
      (19,2) circle(4pt) coordinate (C) node[anchor=180]{$C$}
      (0,-3) circle(4pt) coordinate (P1) node[anchor=45]{$P_{1}$}
      (6,-2) circle(4pt) coordinate (P3) node[anchor=45]{$P_{3}$}
      (4,-8) circle(4pt) coordinate (P4) node[anchor=20]{$P_{4}$}
      (6,-7) circle(4pt) coordinate(P5) node[anchor=-20]{$P_{5}$}
      (7,-5) circle(4pt) coordinate (P6) node[anchor=-20]{$P_{6}$}
      (9,-3) circle(4pt) coordinate (P7) node[anchor=20]{$P_{7}$}
      (10,-6) circle(4pt) coordinate (P8) node[anchor=50]{$P_{8}$}
      (9,0) circle(4pt) coordinate (P9) node[anchor=10]{$P_{9}$}
      (12,-5) circle(4pt) coordinate (P10) node[anchor=20]{$P_{10}$}
      (16,-7) circle(4pt) coordinate (P11) node[anchor=100]{$P_{11}$}
      (10,1) circle(4pt) coordinate (P12) node[anchor=110]{$P_{12}$}
      (13,-1) circle(4pt) coordinate (P13) node[anchor=50]{$P_{13}$}
      (19,-4) circle(4pt) coordinate (P14) node[anchor=130]{$P_{14}$}
      (14,0) circle(4pt) coordinate (P15) node[anchor=130]{$P_{15}$}
      (18,-1) circle(4pt) coordinate (P16) node[anchor=120]{$P_{16}$}
      (20,-1) circle(4pt) coordinate (P17) node[anchor=180]{$P_{17}$}
      (16,1) circle(4pt) coordinate (P19) node[anchor=120]{$P_{19}$};
\draw (6,0) circle(4pt) coordinate (P2) node[anchor=20]{$P_{2}$}
      (12,1.5) circle(4pt) coordinate (P18) node[anchor=160]{$P_{18}$};
\draw (A) -- (B) -- (P1) -- (P4) -- (P11) -- (P14) -- (P17) -- (C) -- (A);
\draw[very thick] (B) -- (P3) -- (P6) -- (P9) -- (P12) -- (P13) -- (P15) -- (P19) -- (C);
\draw (A) -- (P6)
      (A) -- (P13)
      (A) -- (P3)
      (A) -- (P9)
      (A) -- (P12)
      (A) -- (P15)
      (A) -- (P19)
      (P4) -- (P5) -- (P6) -- (P8) -- (P11) -- (P13) -- (P14);

\begin{scope}[on background layer]
        \filldraw [gray!40] (P4) -- (P5) -- (P6) -- (P8) -- (P11) -- (P4);
        \filldraw [gray!40] (P11) -- (P13) -- (P14) -- (P11);
\end{scope}

\end{tikzpicture}

\caption{Illustration of the proof of Lemma~\ref{lemma:main} with $\sigma$
positive at $2$, $6$, $13$ and $18$, and negative in the rest of the $P_i$. The ``white dots'' at $P_2$ and $P_{18}$ indicate points that are not used.
The link $L$ of $A$ is marked thicker. The white polygons above and below $L$ are the $\CR_i$ and $\CP_j$ respectively, and they are part of the regular subdivision $S$ constructed in the proof. The way in which the shaded regions are subdivided in $S$ is not determined.}
\label{fig:step2}
\end{figure}

By inductive hypothesis, we  assume that each polygon $\CP_i$ has at least $C_{n_i}$ regular triangulations, where $n_i$ is its length.
Fix one such regular triangulation $T_i$ for each $\CP_i$. We can assume that the height function producing it
is $0$ for the first and last point in $\CP_i$, which are the only points it has in common with the rest of $\CP_j$'s. Thus, there is a global height function $\omega\in \R^n$ that restricted to each $\CP_i$ produces the regular triangulation $T_i$.

Then, by Lemma~\ref{lemma:refinement}, we can choose independently a regular triangulation for each of the polygons $\CP_i$ to obtain many different regular triangulations of $\CA$ with the given signature. Since each $\CP_i$ has at least $C_{n_i}$ triangulations (by inductive hypothesis), we get at least $\prod_{i}C_{n_i}$ regular triangulations of $\CA$ with signature $\sigma$, where $n_1,\dots,n_k$ are the lengths of the negative intervals of $\sigma$ in $\CA$. 

Now we look at $\CB$. We can consider the polyline $L'$ induced by $\sigma$ in $\CB$, and its negative intervals. Let $m_1,\dots,m_\ell$ be their lengths. As above, we have that $\CB$ has at least $\prod_{j}C_{m_j}$ triangulations with signature $\sigma$; but we can now also argue that the count is exact. Indeed, above the polyline all triangulations with a given signature are the same. Below the polyline what we have are convex polygons of sizes $m_1,\dots,m_\ell$, so the number of ways of refining them to triangulations is exactly 
$\prod_{j}C_{m_j}$.

Thus, $\CA$ has at least $\prod_{i}C_{n_i}$ regular triangulations of signature $\sigma$ and $\CB$ has exactly $\prod_{j}C_{m_j}$ of them.
What remains to be shown is that 
\begin{align}
\prod_{i}C_{n_i} \ge \prod_{j}C_{m_j}.
\label{eq:catalan}
\end{align}
This inequality follows from the following remark: the negative intervals of $\sigma$ in $\CB$ are simply the ``negative intervals of $\sigma$'' in the standard sense; that is, each one starts and ends with a pair of consecutive non-negative entries of $\sigma$. Put differently, the negative intervals with respect to $\CB$, considered as an ordered partition of $|\sigma^-|$, form a refinement of the negative intervals with respect to $\CA$.
In the refinement process the lengths of intervals in $\CA$ are decomposed as sums of lengths of intervals in $\CB$. 
Thus, the inequality \eqref{eq:catalan} follows from applying Lemma~\ref{lemma:catalan} to each negative interval of $\CA$.
\end{proof}

\section{The number of regular subdivisions}
\label{sec:subdivs}

\subsection{Extended signatures and extended stars}

In order to prove Theorem~\ref{thm:main2} we need to extend to arbitrary subdivisions the formalism of link signatures, and introduce several additional concepts.

As in the previous section, we let $\CA$ be an arbitrary point configuration in general position in the plane and with $n$ points, let $B,A, C\in \CA$ be three consecutive vertices of its convex hull and we let $B=P_0,\dots, P_{n-2}=C$ be the list of points in $\CA\setminus \{A\}$, ordered as seen from $A$. We also let $\CB$ be a configuration in convex position obtained moving each point $P_i$ to a new position $Q_i$ along the ray from $A$ through $P_i$.

Let $T$ be an arbitrary subdivision of $\CA$ or $\CB$. We define its \emph{link signature} $\sigma_T\in \{0,+1,-1\}^{[n-3]}$ modifying the definition above as follows: first, the link $L$ of $A$ in $T$ divides the $n-3$ points into points ``above'' and ``below'' $L$, exactly as in Section 2. The points below the link get negative sign in $\sigma_T$. However, the points above can get either positive or zero sign, depending on the following:
\begin{itemize}
\item We give positive sign to the points that either form an edge with $A$ or are not used in the subdivision $T$.
\item We give zero sign to the points that are used but do not form an edge with $A$.
\end{itemize}
See Figure~\ref{fig:extended_star_plus} for an example.

\begin{figure}[htb]
\centering
%

\begin{tikzpicture}[scale=0.5]
\fill (8,3) circle(7pt) coordinate (A) node[anchor=270]{$A$}
      (0,0) circle(4pt) coordinate (B) node[anchor=0]{$B$}
      (19,2) circle(4pt) coordinate (C) node[anchor=180]{$C$}
      (0,-3) circle(4pt) coordinate (P1) node[anchor=45]{$P_{1}$}
      (6,-2) circle(4pt) coordinate (P3) node[anchor=45]{$P_{3}$}
      (4,-8) circle(4pt) coordinate (P4) node[anchor=20]{$P_{4}$}
      (6,-7) circle(4pt) coordinate(P5) node[anchor=-20]{$P_{5}$}
      (7,-5) circle(4pt) coordinate (P6) node[anchor=-20]{$P_{6}$}
      (9,-3) circle(4pt) coordinate (P7) node[anchor=20]{$P_{7}$}
      (10,-6) circle(4pt) coordinate (P8) node[anchor=50]{$P_{8}$}
      (9,0) circle(4pt) coordinate (P9) node[anchor=10]{$P_{9}$}
      (12,-5) circle(4pt) coordinate (P10) node[anchor=20]{$P_{10}$}
      (16,-7) circle(4pt) coordinate (P11) node[anchor=100]{$P_{11}$}
      (10,1) circle(4pt) coordinate (P12) node[anchor=110]{$P_{12}$}
      (13,-1) circle(4pt) coordinate (P13) node[anchor=50]{$P_{13}$}
      (19,-4) circle(4pt) coordinate (P14) node[anchor=130]{$P_{14}$}
      (14,0) circle(4pt) coordinate (P15) node[anchor=130]{$P_{15}$}
      (18,-1) circle(4pt) coordinate (P16) node[anchor=120]{$P_{16}$}
      (20,-1) circle(4pt) coordinate (P17) node[anchor=180]{$P_{17}$}
      (12,1.5) circle(4pt) coordinate (P18) node[anchor=160]{$P_{18}$}
      (16,1) circle(4pt) coordinate (P19) node[anchor=120]{$P_{19}$};
\draw (6,0) circle(4pt) coordinate (P2) node[anchor=20]{$P_{2}$};
\draw (A) -- (B) -- (P1) -- (P4) -- (P11) -- (P14) -- (P17) -- (C) -- (A);
\draw[very thick] (B) -- (P3) -- (P6) -- (P9) -- (P12) -- (P13) -- (P15) -- (P19) -- (C);
\draw (A) -- (P3)
      (A) -- (P9)
      (A) -- (P12)
      (A) -- (P13)
      (A) -- (P15)
      (A) -- (P19)
      (P4) -- (P5) -- (P6) -- (P8) -- (P11) -- (P13) -- (P14);

\begin{scope}[on background layer]
        \filldraw [gray!40] (P4) -- (P5) -- (P6) -- (P8) -- (P11) -- (P4);
        \filldraw [gray!40] (P11) -- (P13) -- (P14) -- (P11);
\end{scope}

\end{tikzpicture}

\caption{An extended star (white region, subdivided into cells), for a $\sigma$ with two positive entries ($P_2$ and $P_{13}$), two zero entries ($P_6$ and $P_{18}$), and negative in the rest of the $P_i$}
\label{fig:extended_star_plus}
\end{figure}

\begin{example}
If $T$ is a triangulation then all  points along but ``above'' $L$ form edges with $A$ and all points strictly above $L$ are unused in $T$, so $\sigma_T$ has no zero entries.
That is, the new definition of link signature for subdivisions is consistent with the one for triangulations in the previous section. 
\end{example}

\begin{example}
For points in convex position, the points above $L$ that are used are exactly the ones that form an edge with $A$. That is, in this case, each point $P_1,\dots,P_{n-3}$ receives a positive sign if $AP_i$ is an edge in $T$; $0$ if $P_i$ is in some cell of $T$ containing $A$ but $AP_i$ is not an edge; and negative sign if $P_i$ is not in a cell with $A$. 
\end{example}

In the following statement, we call \emph{star} of $A$ in a subdivision $T$ the collection of cells of $T$ that contain $A$.

\begin{lemma}
\label{lemma:signatures2}
For every configuration $\CA\subset \R^2$ of $n$ points in general position,
the above rules provide a bijection between the possible stars of $A$ in subdivisions of $\CA$ and their link signatures $\sigma \in \{0,+1,-1\}^{[n-3]}$.
\end{lemma}

\begin{proof}
We know how to construct the link signature from the subdivision; let us see how to recover the star of $A$ in a subdivision $T$ knowing only the link signature $\sigma_T \in \{0,+1,-1\}^{[n-3]}$. 

Taking all zeroes as if they were $+1$ gives us the polyline $L$ from $\sigma_T$. This polyline is the boundary of the star of $A$ and the only extra information that we need is (a) how to partition $L$ into sub-polylines corresponding to the individual cells in the star and (b) which points strictly above the polyline are used or not used in $T$. Both pieces of information are contained in $\sigma_T$:
for (a) we only need to know which vertices of $L$ are joined to $A$ by edges, and these are precisely the positive points along $L$. For (b), the link signature tells us which points strictly above the polyline are to be used (points with zero signature) or not used (points with positive signature) in the subdivision. 
\end{proof}

Notice how, both in the arbitrary and in the convex configurations, each ``$0$'' in the signature acts like a ``$+1$'' in terms of the polyline, but it increases by one the dimension of the subdivision in the secondary polytope (provided that the subdivision is regular). Thus, a regular subdivision with signature $\sigma$ has dimension at least $|\sigma^0|$ as a face of the secondary polytope. Here and elsewhere we denote $\sigma^0$, $\sigma^+$ and $\sigma^-$ the zero, positive, and negative parts of $\sigma$. That is,
\[
\sigma^\varepsilon :=\{i \in [n-3]: \sigma(i) = \varepsilon\}.
\]

Figure~\ref{fig:hexagon} shows the nine coarse subdivisions of a hexagon (corresponding to the nine facets of the three-dimensional associahedron) each with its  signature.

\begin{figure}[htb]
\centering
%

\newcommand{\hexagon}
{
\fill (0,7) circle(7pt) coordinate (A) node[anchor=south]{$A$}
      (3,5) circle(4pt) coordinate (C) node[anchor=west]{$C$}
      (3,2) circle(4pt) coordinate (P3) node[anchor=west]{$P_3$}
      (0,0) circle(4pt) coordinate(P2) node[anchor=north]{$P_2$}
      (-3,2) circle(4pt) coordinate(P1) node[anchor=east]{$P_1$}
      (-3,5) circle(4pt) coordinate(B) node[anchor=east]{$B$};
\draw (C) -- (A) -- (B) -- (P1) -- (P2) -- (P3) -- (C);
}

\begin{multicols}{3}
\begin{tikzpicture}[scale=0.30]
\hexagon 
\draw (A) -- (P1);
\draw [very thick] (B) -- (P1) -- (P2) -- (P3) -- (C);
\end{tikzpicture}
\vskip-22pt
\[
(+1,0,0)
\]
\begin{tikzpicture}[scale=0.30]
\hexagon 
\draw [very thick] (B) -- (P2) -- (P3) -- (C);
\end{tikzpicture}
\vskip-22pt
\[
(-1,0,0)
\]
\begin{tikzpicture}[scale=0.30]
\hexagon
\draw [very thick] (B) -- (P3) -- (C);
\end{tikzpicture}
\vskip-22pt
\[
(-1,-1,0)
\]

\columnbreak
\begin{tikzpicture}[scale=0.30]
\hexagon 
\draw (A) -- (P2);
\draw [very thick] (B) -- (P1) -- (P2) -- (P3) -- (C);
\end{tikzpicture}
\vskip-22pt
\[
(0,+1,0)
\]
\begin{tikzpicture}[scale=0.30]
\hexagon 
\draw [very thick] (B) -- (P1) -- (P3) -- (C);
\end{tikzpicture}
\vskip-22pt
\[
(0,-1,0)
\]
\begin{tikzpicture}[scale=0.30]
\hexagon 
\draw [very thick] (B) -- (C);
\end{tikzpicture}
\vskip-22pt
\[
(-1,-1,-1)
\]

\columnbreak
\begin{tikzpicture}[scale=0.30]
\hexagon 
\draw (A) -- (P3);
\draw [very thick] (B) -- (P1) -- (P2) -- (P3) -- (C);
\end{tikzpicture}
\vskip-22pt
\[
(0,0,+1)
\]
\begin{tikzpicture}[scale=0.30]
\hexagon 
\draw [very thick] (B) -- (P1) -- (P2) -- (C);
\end{tikzpicture}
\vskip-22pt
\[
(0,0,-1)
\]
\begin{tikzpicture}[scale=0.30]
\hexagon 
\draw [very thick] (B) -- (P1) -- (C);
\end{tikzpicture}
\vskip-22pt
\[
(0,-1,-1)
\]

\end{multicols}

\caption{The nine coarse subdivisions of a hexagon, each with its extended signature. The point $A$ is the top one, displayed with a thicker dot.
}
\label{fig:hexagon}
\end{figure}

Figure~\ref{fig:moae} shows the ten regular coarse subdivisions of the so-called ``mother of all examples'', consisting of the vertices of two concentric parallel triangles, each with its extended signature. 

\begin{figure}[htb]
%

\newcommand{\moae}
{
\fill (90:5) circle(7pt) coordinate (A) node[anchor=180]{$A$}
      (210:5) circle(4pt) coordinate (B) node[anchor=-20]{$B$}
      (330:5) circle(4pt) coordinate (C) node[anchor=200]{$C$};
\draw (A) -- (B) -- (C) -- (A);
}

\begin{multicols}{3}
\begin{tikzpicture}[scale=0.30]
\moae
\fill (90:1.5) circle(4pt) coordinate (P2) node[anchor=-90]{$P_2$}
      (330:1.5) circle(4pt) coordinate (P3) node[anchor=200]{$P_3$};
\draw (210:1.5) circle(4pt) coordinate (P1) node[anchor=-20]{$P_1$};
\draw [very thick] (B) -- (C);
\end{tikzpicture}
\vskip-22pt
\[
(+1,0,0)
\]
\begin{tikzpicture}[scale=0.30]
\moae
\fill (210:1.5) circle(4pt) coordinate (P1) node[anchor=-35]{$P_1$}
      (90:1.5) circle(4pt) coordinate (P2) node[anchor=-90]{$P_2$}
      (330:1.5) circle(4pt) coordinate (P3) node[anchor=200]{$P_3$};
\draw (A) -- (P1);
\draw [very thick] (B) -- (P1) -- (C);
\end{tikzpicture}
\vskip-22pt
\[
(-1,0,0)
\]
\begin{tikzpicture}[scale=0.30]
\moae
\fill (210:1.5) circle(4pt) coordinate (P1) node[anchor=225]{$P_1$}
      (90:1.5) circle(4pt) coordinate (P2) node[anchor=-135]{$P_2$}
      (330:1.5) circle(4pt) coordinate (P3) node[anchor=200]{$P_3$};
\draw (A) -- (P2)
      (B) -- (P3);
\draw [very thick] (B) -- (P2) -- (P3) -- (C);
\end{tikzpicture}
\vskip-22pt
\[
(-1,-1,0)
\]
\vfill
\columnbreak
\begin{tikzpicture}[scale=0.30]
\moae
\fill (210:1.5) circle(4pt) coordinate (P1) node[anchor=-20]{$P_1$}
      (330:1.5) circle(4pt) coordinate (P3) node[anchor=200]{$P_3$};
\draw (90:1.5) circle(4pt) coordinate (P2) node[anchor=-90]{$P_2$};
\draw (A) -- (B) -- (C) -- (A);
\draw [very thick] (B) -- (C);
\end{tikzpicture}
\vskip-22pt
\[
(0,+1,0)
\]
\begin{tikzpicture}[scale=0.30]
\moae
\fill (210:1.5) circle(4pt) coordinate (P1) node[anchor=-20]{$P_1$}
      (90:1.5) circle(4pt) coordinate (P2) node[anchor=-135]{$P_2$}
      (330:1.5) circle(4pt) coordinate (P3) node[anchor=200]{$P_3$};
\draw (P1) -- (P3)
      (A) -- (P2);
\draw [very thick] (B) -- (P1) -- (P2) -- (P3) -- (C);
\end{tikzpicture}
\vskip-22pt
\[
(0,-1,0)
\]
\begin{tikzpicture}[scale=0.30]
\moae
\fill (210:1.5) circle(4pt) coordinate (P1) node[anchor=90]{$P_1$}
      (90:1.5) circle(4pt) coordinate (P2) node[anchor=-135]{$P_2$}
      (330:1.5) circle(4pt) coordinate (P3) node[anchor=90]{$P_3$};
\draw (A) -- (P2);
\draw [very thick] (B) -- (P2) -- (C);
\end{tikzpicture}
\vskip-22pt
\[
(-1,-1,-1)
\]
\begin{tikzpicture}[scale=0.30]
\moae
\fill (210:1.5) circle(4pt) coordinate (P1) node[anchor=-20]{$P_1$}
      (90:1.5) circle(4pt) coordinate (P2) node[anchor=-90]{$P_2$}
      (330:1.5) circle(4pt) coordinate (P3) node[anchor=200]{$P_3$};
\draw (A) -- (P1)
      (A) -- (P3);
\draw [very thick] (B) -- (P1) -- (P3) -- (C);
\end{tikzpicture}
\vskip-22pt
\[
(-1,0,-1)
\]

\columnbreak
\begin{tikzpicture}[scale=0.30]
\moae
\fill (90:1.5) circle(4pt) coordinate (P2) node[anchor=-90]{$P_2$}
      (210:1.5) circle(4pt) coordinate (P1) node[anchor=-20]{$P_1$};
\draw (330:1.5) circle(4pt) coordinate (P3) node[anchor=200]{$P_3$};
\draw (A) -- (B) -- (C) -- (A);
\draw [very thick] (B) -- (C);
\end{tikzpicture}
\vskip-22pt
\[
(0,0,+1)
\]
\begin{tikzpicture}[scale=0.30]
\moae
\fill (210:1.5) circle(4pt) coordinate (P1) node[anchor=-20]{$P_1$}
      (90:1.5) circle(4pt) coordinate (P2) node[anchor=-90]{$P_2$}
      (330:1.5) circle(4pt) coordinate (P3) node[anchor=200]{$P_3$};
\draw (A) -- (P3);
\draw [very thick] (B) -- (P3) -- (C);
\end{tikzpicture}
\vskip-22pt
\[
(0,0,-1)
\]
\begin{tikzpicture}[scale=0.30]
\moae
\fill (210:1.5) circle(4pt) coordinate (P1) node[anchor=-30]{$P_1$}
      (90:1.5) circle(4pt) coordinate (P2) node[anchor=-135]{$P_2$}
      (330:1.5) circle(4pt) coordinate (P3) node[anchor=-45]{$P_3$};
\draw (A) -- (P2)
      (C) -- (P1);
\draw [very thick] (B) -- (P1) -- (P2) -- (C);
\end{tikzpicture}
\vskip-22pt
\[
(0,-1,-1)
\]
\vfill
\color{white}{*}
\end{multicols}

\caption{The ten coarse regular subdivisions of ``the mother of all examples'', each with its  signature. Each of the three subdivisions in the first row has an unused point, displayed as a white dot.
}
\label{fig:moae}
\end{figure}

Our proof of Theorem~\ref{thm:main2} is again be stratified by signature. That is, we  show that $\CA$ has at least as many regular subdivisions as $\CB$ for each possible dimension (as a face of the secondary polytope) and signature
$\sigma\in\{-1,0,1\}^{[n-3]}$. 
This can be seen in Figures~\ref{fig:hexagon} and \ref{fig:moae}: the former contains nine subdivisions with nine different signatures and the latter contains these same nine plus an additional tenth signature.

In fact, the proof is  stratified by \emph{extended star} according to the following definitions. 

\begin{definition}[Negative intervals, extended star]
\label{defi:intervals}
Let $\sigma\in \{0,+1,-1\}^{[n-3]}$ be a signature and $L$ its corresponding polyline in a configuration $\CA$. 

We call negative intervals of $\sigma$ (with respect to $\CA$) the subsets 
\[
\{i,j\}\cup [i,j] \cap \sigma^- ,
\]
where $i$ and $j$ are two consecutive non-negative entries of $\sigma$ along the polyline. 

The \emph{extended star} induced by $\sigma$ in $\CA$ consists of the following two types of cells:
\begin{enumerate}
    \item The cells in the ``star of $A$'' corresponding to $\sigma$, as given by the bijection of Lemma~\ref{lemma:signatures2}. We call these the \emph{cells above $L$}.
    \item The negative intervals. We call these the \emph{cells below $L$}.
\end{enumerate}
\end{definition}

\begin{remark}
The cells of the extended star above $L$ could well be called ``zero intervals'', since each of them consists of the zero entries in $\sigma$ between any two consecutive points $P_i$ and $P_j$ in $L$ with non-zero signature (including $A$ and $P_i$ and $P_j$ as elements of the interval).

If $T$ is a triangulation, these points include all the points along $L$, so the cells above $L$ are exactly the triangles $\CR_j$ in the proof of Lemma~\ref{lemma:main}.
\end{remark}

Observe that the two types of cells in the extended star form two ordered sequences as seen from $A$. In the cells above $L$, each cell shares with the next one an edge of the form $AP_i$. In the cells below $L$, each cell shares with the next one a single point (the last point of one negative interval, which coincides with the first point of the next).

For a configuration in convex position (say $\CB$) the cells in the extended star cover the whole convex hull. That is to say, they form a polyhedral subdivision of the configuration. For a general configuration this is not the case. Still:

\begin{lemma}
\label{lemma:extended_star}
Let $\CA$ be a configuration in general position and let $\sigma\in \{0,+1,-1\}^{[n-3]}$ be a signature. Then, the extended star $S$ induced by $\sigma$ in $\CA$ can be extended to a regular subdivision (that is, there is a regular subdivision of $\CA$ containing $S$).
\end{lemma}

\begin{proof}
As usual, let $B$ and $C$ be the vertices of $\conv(\CA)$ adjacent to $A$ and let $B=P_{0}$, \dots, $P_{n-2}=C$ be the points of $\CA \setminus A$ in the order they are seen from $A$. Let $B=P_{a_0},P_{a_1},\dots, P_{a_{s-1}}, P_{a_s}=C$ be the points along $L$.

%
%

For the time being we assume that $\sigma$ has no positive entry, that is, it uses only $0$ or $-1$. This has the effect that every internal vertex of the polyline $L$ belongs to exactly three cells of $S$, two above $L$ and one below if the point has negative signature, and one above and two below if it has zero signature. Moreover, for each three consecutive ones $P_{a_{i-1}}P_{a_i}P_{a_{i+1}}$ there is a unique cell $\CR_i \in S$ that contains the three of them. Observe that sometimes $\CR_i=\CR_{i+1}$ (e.g., the cell containing $P_6 P_9 P_{12} P_{13}$ in Figure~\ref{fig:extended_star}) and some cells of $S$ are not an $\CR_i$ for any $·i$ (e.g., the cells above $P_9P_{13}$ or above $P_{15}P_{19}$). 

We consider a second copy $A'$ of the point $A$. This point $A'$ is used as a tool to define the heights for the rest, but it is not  part of the configuration.
Give height $+1$ to $A$, $-1$ to $A'$, and arbitrary (e.g., $0$) to $B=P_{a_0}$ and to $P_{a_1}$. Once this is done, give every other point $P_{a_{i+1}}$, $i\in \{1,\dots,s-1\}$ along $L$  the height that lifts $P_{a_{i-1}}P_{a_i}P_{a_{i+1}}$ coplanar to $A$ if $\CR_i$ is above $L$ and coplanar to $A'$ if $\CR_i$ is below $L$. Finally, for each point $P$ not in $L$ consider an edge $e=P_{a_i}P_{a_{i+1}}$ along $L$ belonging to the unique cell of $S$ containing $P$. Lift $P$ to lie coplanar to $e$ and $A$ (resp. to $e$ and $A'$) if the cell  containing $P$ is above (resp. below) $L$.


\begin{figure}[htb]
\centering
%

\begin{tikzpicture}[scale=0.5]
\fill (8,3) circle(7pt) coordinate (A) node[anchor=270]{$A$}
      (0,0) circle(4pt) coordinate (B) node[anchor=0]{$B$}
      (19,2) circle(4pt) coordinate (C) node[anchor=180]{$C$}
      (0,-3) circle(4pt) coordinate (P1) node[anchor=45]{$P_{1}$}
      (6,0) circle(4pt) coordinate (P2) node[anchor=20]{$P_{2}$}
      (6,-2) circle(4pt) coordinate (P3) node[anchor=45]{$P_{3}$}
      (4,-8) circle(4pt) coordinate (P4) node[anchor=20]{$P_{4}$}
      (6,-7) circle(4pt) coordinate(P5) node[anchor=-20]{$P_{5}$}
      (7,-5) circle(4pt) coordinate (P6) node[anchor=-20]{$P_{6}$}
      (9,-3) circle(4pt) coordinate (P7) node[anchor=20]{$P_{7}$}
      (10,-6) circle(4pt) coordinate (P8) node[anchor=50]{$P_{8}$}
      (9,0) circle(4pt) coordinate (P9) node[anchor=10]{$P_{9}$}
      (12,-5) circle(4pt) coordinate (P10) node[anchor=20]{$P_{10}$}
      (16,-7) circle(4pt) coordinate (P11) node[anchor=100]{$P_{11}$}
      (10,1) circle(4pt) coordinate (P12) node[anchor=110]{$P_{12}$}
      (13,-1) circle(4pt) coordinate (P13) node[anchor=50]{$P_{13}$}
      (19,-4) circle(4pt) coordinate (P14) node[anchor=130]{$P_{14}$}
      (14,0) circle(4pt) coordinate (P15) node[anchor=130]{$P_{15}$}
      (18,-1) circle(4pt) coordinate (P16) node[anchor=120]{$P_{16}$}
      (20,-1) circle(4pt) coordinate (P17) node[anchor=180]{$P_{17}$}
      (12,1.5) circle(4pt) coordinate (P18) node[anchor=160]{$P_{18}$}
      (16,1) circle(4pt) coordinate (P19) node[anchor=120]{$P_{19}$};
\draw (A) -- (B) -- (P1) -- (P4) -- (P11) -- (P14) -- (P17) -- (C) -- (A);
\draw[very thick] (B) -- (P3) -- (P6) -- (P9) -- (P12) -- (P13) -- (P15) -- (P19) -- (C);
\draw (A) -- (P3)
      (A) -- (P9)
      (A) -- (P12)
      (A) -- (P15)
      (A) -- (P19)
      (P4) -- (P5) -- (P6) -- (P8) -- (P11) -- (P13) -- (P14);

\begin{scope}[on background layer]
        \filldraw [gray!40] (P4) -- (P5) -- (P6) -- (P8) -- (P11) -- (P4);
        \filldraw [gray!40] (P11) -- (P13) -- (P14) -- (P11);
\end{scope}

\end{tikzpicture}

\caption{An extended star (white region, subdivided into cells), for a $\sigma$ with no positive entries: zero at $2$, $6$, $13$ and $18$, and negative in the rest of the $P_i$}
\label{fig:extended_star}
\end{figure}


Observe that this lift is essentially unique. The only choices were the heights of the first four points $A$, $A'$, $B=P_0=P_{a_0}$ and $P_{a_1}$, which can be arbitrarily changed by adding an affine global function to the heights, as long as this function does not change the orientation of the lifted tetrahedron that these four points form.

By definition, these heights lift each cell of $S$ coplanar, and we only need to check that it lifts them as lower facets of the lifted configuration. For this, let us consider the following unbounded cell complex $S'$, depicted in Figure~\ref{fig:extended_star_rays}: forget all the edges of $S$ below $L$ and insert instead rays separating each two consecutive lower cells of $S$. (These additional rays can be alternatively described as the rays starting at zero points in $L$, including $B$ and $C$, and going in the direction opposite to $A$). Since lower cells are lifted coplanar to $A'$, each additional ray is lifted coplanar to the two lower cells incident to it, so we can think of the heights as a lift of $S'$, and want to check that it is a convex lift. This is equivalent to checking that every edge or ray of $S'$ is lifted convex.

For the edges lying in $L$ convexity follows from the fact that the cell above $L$ is coplanar to $A$ (in fact, it contains $A$) and the cell below $L$ is coplanar to $A'$, with $A'$ lifted below $A$. For the rest of edges/rays convexity follows from the fact that  every internal vertex of $S'$ has degree three: since one of the three edges at that edge (the one from $L$) is lifted convex, the other two are lifted convex too.

\begin{figure}[htb]
\centering
%

\begin{tikzpicture}[scale=0.5]
\fill (8,3) circle(7pt) coordinate (A) node[anchor=270]{$A$}
      (0,0) circle(4pt) coordinate (B) node[anchor=-30]{$B$}       (-8/3,-1) circle(0pt) coordinate (B') node[anchor=0]{}
      (19,2) circle(4pt) coordinate (C) node[anchor=210]{$C$}       (22.3,1.7) circle(0pt) coordinate (C') node[anchor=180]{}
      (0,-3) circle(4pt) coordinate (P1) node[anchor=45]{$P_{1}$}
      (6,0) circle(4pt) coordinate (P2) node[anchor=20]{$P_{2}$}
      (6,-2) circle(4pt) coordinate (P3) node[anchor=45]{$P_{3}$}
      (4,-8) circle(4pt) coordinate (P4) node[anchor=20]{$P_{4}$}
      (6,-7) circle(4pt) coordinate(P5) node[anchor=-20]{$P_{5}$}
      (7,-5) circle(4pt) coordinate (P6) node[anchor=-20]{$P_{6}$}       (6.5,-9) circle(0pt) coordinate (P6') node[anchor=-20]{}
      (9,-3) circle(4pt) coordinate (P7) node[anchor=20]{$P_{7}$}
      (10,-6) circle(4pt) coordinate (P8) node[anchor=50]{$P_{8}$}
      (9,0) circle(4pt) coordinate (P9) node[anchor=10]{$P_{9}$}
      (12,-5) circle(4pt) coordinate (P10) node[anchor=20]{$P_{10}$}
      (16,-7) circle(4pt) coordinate (P11) node[anchor=100]{$P_{11}$}
      (10,1) circle(4pt) coordinate (P12) node[anchor=110]{$P_{12}$}
      (13,-1) circle(4pt) coordinate (P13) node[anchor=50]{$P_{13}$}       (20.5,-7) circle(0pt) coordinate (P13') node[anchor=50]{}
      (19,-4) circle(4pt) coordinate (P14) node[anchor=130]{$P_{14}$}
      (14,0) circle(4pt) coordinate (P15) node[anchor=130]{$P_{15}$}
      (18,-1) circle(4pt) coordinate (P16) node[anchor=120]{$P_{16}$}
      (20,-1) circle(4pt) coordinate (P17) node[anchor=180]{$P_{17}$}
      (12,1.5) circle(4pt) coordinate (P18) node[anchor=160]{$P_{18}$}
      (16,1) circle(4pt) coordinate (P19) node[anchor=120]{$P_{19}$};
\draw (A) -- (B) ;
\draw (C) -- (A);
\draw[very thick] (B) -- (P3) -- (P6) -- (P9) -- (P12) -- (P13) -- (P15) -- (P19) -- (C);
\draw (A) -- (P3)
      (A) -- (P9)
      (A) -- (P12)
      (A) -- (P15)
      (A) -- (P19);
\draw[dashed] (P6) -- (P6');
\draw[dashed] (P13) -- (P13');
\draw[dashed] (B) -- (B');
\draw[dashed] (C) -- (C');


\end{tikzpicture}

\caption{Forgetting the edges below $L$ and inserting rays away from the  points of $L$ with zero signature, to show that our lift is regular.}
\label{fig:extended_star_rays}
\end{figure}


This finishes the proof of the Lemma under the assumption that $\sigma$ has no positive entries. The modification for positive entries is easy:
\begin{itemize}
    \item Positive points of $\sigma$ that are strictly above $L$ correspond to unused points in the extended star. Such points do not affect regularity, since we can give them any sufficiently big positive height.
    \item Positive points of $\sigma$ that lie in $L$ correspond to additional edges incident to $A$. In a first phase take those signs as if they were zero and construct the heights for a regular subdivision extending $S$ as explained above. In a second phase, iteratively perturb the chosen heights, in the sense of  Lemma~\ref{lemma:refinement}, as follows, processing the positive points along $L$ in an arbitrary order: For each such point $P$, let $r_P$ be the straight line containing $AP$ and 
    consider the lifting height $\omega_P$ that lifts each point of $\CA$ to height equal to its distance to $r_P$. All cells of $S$ except the positive one containing the segment $AP$ lie completely on one side of $r_P$ and, thus, they are lifted planarly. The cell containing  $AP$ crosses $r_P$, so that this one is the only cell of $S$ refined by the perturbation, and its refinement is precisely what we want: the edge $AP$ is introduced in the subdivision.
    \qedhere
\end{itemize}
\end{proof}

\subsection{Well-formed subdivisions}

We now want to take the dimension of a regular subdivision (as a face of the secondary polytope) into consideration. We first analyze the case of points in convex position.

Let $T$ be a subdivision of $\CB$. Let $\sigma\in \{0,+1,-1\}^{[n-3]}$ be its link signature. As mentioned above, the extended star $S$ of $\sigma$ in $\CB$ is a polyhedral subdivision and moreover:

\begin{itemize}
    \item $T$ refines $S$, and
    \item $T$ and $S$ coincide above the link $L$ of $A$.
\end{itemize}

Let $\ell$ be the number of cells in $S$ below $L$ and call $\CP_1,\dots, \CP_\ell$ those cells (observe that, as sets of labels, they are simply the negative intervals in $\sigma$). Let $m_i$ be the length of the $i$-th negative interval, so that  $|\CP_i|=m_i+2$. In $T$, each $\CP_i$ gets subdivided into a subdivision $T_i$. Since $\CP_i$ is a convex polygon, $T_i$ is a regular subdivision and it has a well-defined dimension $\delta_i$. 

\begin{definition}
We call $(\sigma, \delta)$ (with $\delta=(\delta_1,\dots,\delta_\ell)$) the \emph{extended signature} of the subdivision $T$ of $\CB$. 
We say it is a signature of length $n-3$ (the length of $\sigma$).
\end{definition}

The possible extended signatures corresponding to a signature $\sigma$ are easy to characterize. 
First, the lengths of the negative intervals in $\sigma$ give us the numbers $m_1,\dots,m_\ell$. Then, 
$\delta_i$ can be any number from $[0,1,\dots,m_i-1]$,  for each $i=1,\dots, \ell$.

\begin{lemma}
\label{lemma:ngon}
Let $(\sigma,\delta)$ be an extended signature of length $n-3$. Let $m_1,\dots,m_\ell$ be the list of lengths of the negative intervals in $\sigma$. Then, 
\begin{enumerate}
\item The number of subdivisions of the convex $n$-gon with that extended signature equals
\[
\prod_{i=1}^\ell C_{m_i}^{\delta_i},
\]
\item All such subdivisions correspond to faces of dimension $\sum_i \delta_i + |\sigma^0|$ in the secondary polytope.
\end{enumerate}
\end{lemma}

\begin{proof}
The extended star covers the whole $\conv(\CB)$, and all subdivisions of $\conv(\CB)$ are regular. Moreover, the dimension of a subdivision $T$ in the associahedron equals the sum of dimensions of the individual cells.

Thus, both parts of the statement follow from counting in how many ways we can independently subdivide the cells in the extended star. For the cells above $L$ we have no choice (since $(\sigma,\delta)$ fixes the star of $A$), and their combined dimension equals $|\sigma_0|$. The $i$-th cell below is a convex $(m_i+2)$-gon, which has exactly $C_{m_i}^{\delta_i}$ subdivisions of dimension $\delta_i$. 
\end{proof}

For later use we  mention the following inequality among face numbers of associahedra, which follow from partitioning the $(m_1+\ldots+m_\ell+2)$-gon into $m$ convex polygons of sizes $m_1+2, \dots, m_\ell+2$ and then subdividing these independently:
\begin{lemma}
\label{lemma:catalan2}
For nonnegative integers $m_1,\ldots,m_\ell$ and $d$, we have 
\[
C^d_{m_1+\dots+m_\ell} \geq \sum_{d_1+\dots+d_\ell=d} C_{m_1}^{d_1}\cdots C_{m_\ell}^{d_\ell}
\]
\end{lemma}

We now want to extend this formalism of extended signatures (that is, the vector $\delta$ of dimensions associated to the negative intervals in $\sigma$) to \emph{some} regular subdivisions of an arbitrary point configuration $\CA$, that we call well-formed. 
The fact that we do not extend it to \emph{all} subdivisions is not a loss of generality, since what we show is that the well-formed subdivisions are sufficiently many to prove Theorem~\ref{thm:main2}.

\begin{definition}
Let $T$ be a regular subdivision of $\CA$. Let $\sigma$ be its signature and $S$ be the corresponding extended star, with $\CP_1,\dots,\CP_k$ the cells of $S$ below the polyline $L_\sigma$. Remember that $T$ and $S$ coincide above $L_\sigma$.
We say that $T$ is well-formed if it satisfies the following conditions:
\begin{enumerate}
    \item $T$ refines the extended star $S$ below the polyline; that is, $T$ contains a polyhedral subdivision $T_i$ of each cell $\CP_i$.
    \item The dimension of $T$ (as a regular subdivision of $\CA$) equals $|\sigma^0| + \sum_i \delta_i$, where $\delta_i$ is the dimension of $T_i$.
\end{enumerate}
In these conditions, we call $(\sigma, \delta)$ the \emph{extended signature} of $T$.
\end{definition}

\begin{remark}
For a given $\sigma$, the possible dimension vectors $(\delta_1,\dots, \delta_k)$ for the extended signatures are those  with $\delta_i < n_i$, where $(n_1,\dots,n_k)$ is the sequence of lengths of negative intervals of $\sigma$ in $\CA$. We emphasize that the latter depend not only on $\sigma$, but also on the configuration $\CA$, as already happened for triangulations. 

More precisely, if $(n_1,\dots,n_k)$ and 
$(m_1,\dots,m_\ell)$ are the sequences of lengths of  negative intervals of $\sigma$ in $\CA$ and $\CB$ respectively, we have that both sequences are ordered partitions of $|\sigma^-|$ and the latter refines the former.
\end{remark}

The key statement that we want to prove to derive Theorem~\ref{thm:main2} is:

\begin{lemma}
\label{lemma:extension}
Let $\sigma\in \{0,+1,-1\}^{[n-3]}$ be an extended signature for $\CA$, let $S$ be the extended star induced by it, let $\CP_1,\dots, \CP_k$ be the cells of $S$ below  $L$, and let $\delta= (\delta_1,\dots,\delta_k)$ be a valid vector of dimensions (that is, $\delta_i < n_i$ for every $i$, where $n_i = |\CP_i|-2$).

For each $\CP_i$, let $T_i$ be a regular subdivision of $\CP_i$ of dimension $\delta_i$. Then, there is a well-formed regular subdivision $T$ of $\CA$ with signature $(\sigma,\delta)$ and such that $T$ restricted to $\CP_i$ equals $T_i$.
\end{lemma}

We postpone the proof of Lemma~\ref{lemma:extension} and first show how to derive Theorem~\ref{thm:main2} from it.

\begin{corollary}
\label{coro:extension}
With the same notation as in Lemma~\ref{lemma:extension}, the number of well-formed subdivisions of $\CA$ with extended signature $(\sigma, \delta)$ is at least
\[
\prod_{i=1}^k C_{n_i}^{\delta_i}.
\]
\end{corollary}

\begin{proof}[Proof of Corollary~\ref{coro:extension} and Theorem~\ref{thm:main2}]
The proof is by induction. We first prove Corollary~\ref{coro:extension} for configurations of size $n$ assuming Theorem~\ref{thm:main2} for smaller configurations, and then derive from it 
Theorem~\ref{thm:main2} for size $n$. 

The first part is easy. Lemma~\ref{lemma:extension} says that there are at least as many well-formed subdivisions of $\CA$ with extended signature $(\sigma, \delta)$ as there are choices of regular polyhedral subdivisions of dimensions $\delta_i$ for the cells $\CP_i$ of $S$ below $L$. The inductive hypothesis says that the latter are at least $C_{n_i}^{\delta_i}$ for each $i$.

To derive Theorem~\ref{thm:main2}, remember that there is a well-defined map that sends each  extended signature $(\sigma,\delta')$ in $\CB$ to an extended signature $(\sigma,\delta)$ in $\CA$, obtained by grouping together in $\delta$ the entries of $\delta'$ that correspond to the same negative interval of $\sigma$ in $\CA$. To emphasize this map let us slightly change the notation used so far. Let $(n_1,\dots,n_k)$ be the length-sequence of negative intervals of $\sigma$ in $\CA$, and denote 
$(m_1^1,\dots,m_1^{\ell_1}, \dots, m_k^1,\dots, m_k^{\ell_k})$ the sequence in $\CB$, so that
\[
n_i = \sum_{j=1}^{\ell_i} m_i^j
\qquad\text{for each $i=1,\dots,k$.}
\]
Similarly, we use the notation $(\gamma_1^1,\dots,\gamma_1^{\ell_1}, \dots, \gamma_k^1,\dots$, $\gamma_k^{\ell_k})$ for a dimension sequence in $\CB$ compatible with $(m_1^1,\dots,m_1^{\ell_1}, \dots, m_k^1,\dots$, $m_k^{\ell_k})$.

Then, by Lemma~\ref{lemma:ngon} the number of subdivisions of $\CB$ whose  extended signature maps to a fixed extended signature $(\sigma,\delta)$ of $\CA$ equals
\[
\prod_{i=1}^k \left(\sum_{\sum_{j=1}^{\ell_i}\gamma_i^j=\delta_i } \left(\prod_{j=1}^{\ell_i} C_{m_i^j}^{\gamma_i^j}
\right)\right),
\]
where the sum in the middle is over all non-negative tuples $(\gamma_i^1,\dots,\gamma_i^{\ell_i})$ adding up to $\delta_i$.

By Lemma~\ref{lemma:catalan2} this is less or equal than 
\[
\prod_{i=1}^k C_{\sum_j m_i^j}^{\delta_i}=
\prod_{i=1}^k C_{n_i}^{\delta_i}
\]
and, by Corollary~\ref{coro:extension}, this is 
less or equal than the number of regular subdivisions of $\CA$ with signature $(\sigma, \delta)$.
\end{proof}

To complete our job we need to prove Lemma~\ref{lemma:extension}.

\begin{proof}[Proof of Lemma~\ref{lemma:extension}]
We assume that $S$ uses all the points in $\CA$. This is no loss of generality because unused points do not affect regularity and each of them adds one both to the count $|\sigma^0|+\sum_i \delta_i$ and to the dimension of the corresponding regular subdivision of $\CA$.

Let $T_0$ be a regular subdivision that extends $S$, which exists by Lemma~\ref{lemma:extended_star}, and let $\alpha=(\alpha_1,\dots,\alpha_n)$ be its corresponding height vector. By Lemma~\ref{lemma:refinement}, 
for any choice of a perturbing height vector $\omega=(\omega_1,\dots,\omega_n)$ the subdivision obtained refining $T_0$ as prescribed by $\omega$ (restricted to each individual cell) is regular. 

Thus, we need to show that there exists an $\omega$ that restricted to each $\CP_i$ produces the regular subdivision $T_i$. This follows from the same arguments as in Lemma~\ref{lemma:extended_star}: as we process the cells in the extended star in their monotone order and assign heights $\omega_i$ to the points, each cell has at most three points in common with the previously processed ones, so that \emph{any} regular subdivision of the cell can be realized by $\omega$.

The only thing that needs to be shown is that the regular subdivision $T$ so obtained has exactly dimension $|\sigma^0|+\sum_i \delta_i$ in the secondary polytope. Observe that actually $T$ is not uniquely defined. Since $S$ does not cover $\conv(\A)$, different choices of $\omega$ may produce different subdivisions in those uncovered regions. Our claim is that as long as $\omega$ is chosen sufficiently generic (among the possible $\omega$'s with $T|_{\CP_i} = T_i$), we get that $T$ has the desired dimension. 

More precisely: let $\{\CR_1,\dots,\CR_l\}$ be the cells of $S$ above $L$, in order. As we process each cell $\CC$ (be it a $\CP_i$ or a $\CR_i$) in the order specified in Lemma~\ref{lemma:extended_star}
there is a certain set $\omega_J$ ($J\subset[n]$) of coordinates that we are going to fix in this step, and our constraint is that $\omega_J$ needs to lie in the corresponding secondary cone of the configuration $\CC$. The dimension of this cone is $|\CC|-\delta-k$, where $\delta$ is the dimension, in the secondary polytope of $\CC$, of the subdivision that we want (which is the trivial subdivision for the cells above $L$, and the subdivision $T_i$ for the cells below) and $k$ is the number of points of $\CC$ whose $\omega$ was already fixed by the previous cells. Let us count these parameters in each case:
\begin{enumerate}
    \item There is an initial cone (in fact a linear plane) of dimension two corresponding to the fact that we can choose the heights of $A$ and $B$ arbitrarily, before processing any cell.
    \item If $\CC=\CR_i$ is a cell above $L$, then we want the trivial subdivision in it, whose dimension in the secondary polytope is $|\CR_i|-3$. The number of points of $\CR_i$ that were already processed was exactly $3$, so that the dimension of the cone we are looking at is zero. This simply reflects the fact that we have no choice for the lift of $\CR_i$, as happened in Lemma~\ref{lemma:extended_star}.
    
    \item If $\CC=\CP_i$ is a cell below $L$, then the dimension of the subdivision $T_i$ equals $\delta_i$, $|\CP_i| = n_i+2$, and $k$ equals $1$ or $2$. More precisely, we have $k=1$ if and only if the initial point in $\CP_i$ forms an edge with $A$; that is, if and only if it has positive signature. Thus, the dimension of the cone equals $n_i-\delta_i +1$ if $\CP_i$ starts with a positive point and it equals $n_i-\delta_i$ if $\CP_i$ starts with a zero point. In this count we need to consider $B$ as an extra positive point, since $\CP_1$ has only one point in common (the point $B$) with the part that is processed before it (the points $A$ and $B$).
 \end{enumerate}
    This implies that the global contribution of all cells comes from the two initial points and the cells below $L$, and it equals
    \begin{align*}
    & 2 + \sum n_i -\sum \delta_i + |\sigma^+| +1\\=&
    2+ |\sigma^-| -\sum \delta_i + |\sigma^+| +1\\=&
    2+ (n-3 - |\sigma^0|) - \sum \delta_i +1\\=&
    n - (|\sigma^0| + \sum \delta_i).
    \end{align*}
    That is, if a $\omega$ is sufficiently generic among the ones that refine $S$ in the way we want, then the secondary cone of the subdivision so obtained has dimension $n - (|\sigma^0| + \sum \delta_i)$. (Observe here that $\omega$ being sufficiently generic implies it to be in the relative interior of the secondary cone of $T$).
    Since
    the dimension of a regular subdivision of $\CA$ equals $n$ minus the dimension of the corresponding secondary cone. This finishes the proof.
\end{proof}

\section{Some remarks in higher dimension}
\label{sec:higherdim}

For a given dimension $d$ and number of points $n$, what is the $d$-dimensional configuration of size $n$ minimizing the number of triangulations?

Although this question is probably too difficult to be answered explicitly for every $d$ and $n$, we here include several remarks regarding it.

\subsubsection*{Regular versus non-regular triangulations}

In two dimensions, the configuration minimizing the set of all triangulations is the same as the one minimizing the number of regular ones and, in fact, it is a configuration that has \emph{only} regular triangulations.

Moreover, the numbers of regular and non-regular triangulations in the plane for a particular configuration (or for arbitrary configurations) are not that different. They both have upper and lower bunds of the type $k^n$, for constants $k>1$.

In higher dimension several things change drastically:

\begin{enumerate}
    \item For every $d\ge 3$ there is a constant $N$ such that \emph{every} configuration of size $N$ or more in general position in $\R^d$ has non-regular triangulations. This follows from the combination of two facts: the cyclic polytope $C_d(n)$ has non-regular triangulations for every $n\ge \max(d+6,9)$~\cite[Theorem 4.1]{bauescyclic}, and there is a constant $N=N(n,d)$ such that every configuration of more than $N$ points in general position in $\R^d$ contains a subconfiguration isomorphic to the vertex set of $C_d(n)$ (this is called the ``higher-dimensional Erd\H{o}s-Szekeres Theorem'' in~\cite[Proposition 9.4.7]{OMbook}).
    
    \item The number of regular triangulations of any configuration is bounded above by $2^{\Theta(n\log n)}$~\cite[Theorem 8.4.2]{triangbook}, while the number of non-regular ones can be much higher; for $C_d(n)$ it is bounded below by $2^{\Omega(n^{\lfloor d/2\rfloor})}$ \cite[Theorem 6.1.22 and Theorem 8.4.3]{triangbook}.
\end{enumerate}

\subsubsection*{Cyclic polytopes}
The natural candidate generalizing the convex $n$-gon to dimension $d$ is the vertex set of a cyclic $d$-polytope with $n$ vertices, mentioned above. Recall that the cyclic polytope $C_d(n)$ is defined as the convex hull of $n$ arbitrary points along the $d$-dimensional moment curve. The number of triangulations of it is independent of the points chosen (since the oriented matroid is fixed) but the number of regular triangulations is not. See, for example, \cite{AzaSan}, where these numbers are computed quite explicitly for the case $n=d+4$:
\begin{itemize}
    \item The total number of triangulations of $C_{n-4}(n)$ is in $\Theta(n2^n)$~\cite[Theorem 1]{AzaSan}.
    \item The number of regular ones is in the order of $\frac{n^4}{64} \pm\Theta(n^3)$, with the cubic term depending on the specific realization~\cite[Theorem 4.3 and Remark 4.4]{AzaSan}.
\end{itemize}

However, cyclic polytopes are  typically used as examples of polytopes with \emph{many} triangulations or subdivisions; as we already noted, they have $2^{\Omega(n^{\lfloor d/2\rfloor})}$ of them. This, looked from the distance, does not seem too far from the upper bound of $2^{O(n^{\lceil d/2\rceil}\log n)}$ that we know for the number of triangulations of \emph{any} point configuration~\cite[Theorem 8.4.2.1]{triangbook}.

 Thus, it would be extremely surprising if cyclic polytopes turn out to {minimize} the number of triangulations, as the $n$-gon does in the plane. They might, however, minimize the number of \emph{regular ones}, as we now see in a particular case.

\subsubsection*{The case $n=d+4$}

It is easy to show that every configuration of $n=d+2$ points in general position has exactly two triangulations and with $n=d+3$ it has $n$ of them, all regular. The secondary polytopes are, respectively, a segment and an $n$-gon~\cite[Section 5.5.1]{triangbook}.

The next case, configurations of size $n=d+4$, is more complicated but still tractable via Gale transforms~\cite[Sect. 6.4]{Ziegler}. 

For the purposes of this discussion, the Gale transform of a configuration $\CA$ of $n$ points in dimension $d$ is a configuration $\CA^*$ of $n$ points in the sphere of dimension $n-d-2$; that is, in the $2$-sphere for $n=d+4$. For example, the Gale transform of the cyclic polytope $C_{n-4}(n)$ can be realized by placing $\lceil n/2\rceil$ points in a small circle around the north pole and the other $\lfloor n/2\rfloor$ in a small circle around the south pole, in a regular manner~(see, e.g., \cite[Ex. 6.13]{Ziegler} and \cite[p. 262]{triangbook}).

In the case $n=d+4$ there is an easy recipe to compute or count regular triangulations of $\CA$ from its Gale transform $\CA^*$: draw the $\binom{n}2$ (shorter) geodesic arcs joining the $n$ points of $\CA^*$ in the sphere $S^2$, and the regular triangulations of $\CA$ turn out to be in bijection with the 2-dimensional regions cut by these geodesics. This is an instance of~\cite[Corollary 5.4.9]{triangbook}; the cell decomposition of the 2-sphere produced by the $\binom{n}2$ arcs is the \emph{chamber complex} of $\CA^*$.

Let us make the simplifying assumption that the $\binom{n}2$ geodesics do not produce triple crossings. This is a genericity condition that can always be attained via a small perturbation of the points in $\CA^*$ (or, equivalently, of the points in $\CA$, since $\CA$ and $\CA^*$ depend continuously on one another). Then, a simple application of Euler's formula gives the following relation between the number $t$ of regions in the chamber complex of $\CA^*$ (that is, the number of regular triangulations of $\CA$) and the number $c$ of crossings among the arcs in the Gale transform (see Lemma 4.1 in~\cite{AzaSan}):
\[
t = c + \binom{n}2 -n +2.
\]

Thus, deciding what is the configuration of $n$ points in dimension $n-4$ that minimizes the number of regular triangulations (under our genericity assumption) is equivalent to answering the following question:

\begin{question}[Spherical crossing number of $K_n$]
What is the geodesic embedding of the complete graph $K_n$ in the 2-sphere that minimizes the number $c$ of pairs of edges that cross each other?
\end{question}

This is a classical question in geometric/topological graph theory, for which the complete answer is unknown, despite considerable efforts. A summary of what we know is:

\begin{enumerate}
    \item It is conjectured that the minimum value of $c$ that can be attained is
    \[
    Z(n):=\frac{1}{4}
    \left\lfloor \frac{n}2\right\rfloor
    \left\lfloor \frac{n-1}2\right\rfloor
    \left\lfloor \frac{n-2}2\right\rfloor
    \left\lfloor \frac{n-3}2\right\rfloor,
    \]
    not only in geodesic drawings but actually in any topological drawing of $K_n$ in the 2-sphere (or, equivalently, in the plane). This was originally conjectured by Hill and popularized by Guy \cite{Guy,HaHi}.
    
    \item  Several embeddings attaining precisely that number are known, and one of them happens to be the geodesic embedding with points in two opposite circles, that is, the Gale transform of the cyclic polytope. This is  one of the embeddings originally found by Hill, see e.g., \cite[Figure 5]{HaHi}; its number of crossings is also computed, in connection to regular triangulations of $C_{n-4}(n)$, in \cite[Proposition 4.2]{AzaSan}.
    
\end{enumerate}
See~\cite{BeWi} for an account of the early history of the crossing number problem and its variants, and~\cite{Schaefer} for a comprehensive survey.

As a conclusion we have that:

\begin{corollary}
If Hill's Conjecture holds then the minimum number of regular triangulations among all generic configurations of size $n$ and dimension $n-4$ is attained by the vertex set of a (generic) cyclic polytope $C_{n-4}(n)$. The number is  $Z(n) + \binom{n}2 -n +2$.
\end{corollary}

Here ``generic'' is stronger than general position; it can be defined via the property that no three geodesic arcs in the Gale transform $\CA^*$ meet at a point or, also, calling a configuration $\CA$ generic if any sufficiently small perturbation preserves its set of regular triangulations.


\begin{thebibliography}{9}

\bibitem[AH+06]{upper}
Oswin Aichholzer, Thomas Hackl, Clemens Huemer, Ferran Hurtado, Hannes Krasser, and Birgit Vogtenhuber. 
On the number of plane graphs. 
In \emph{Proceedings of the ACM-SIAM Symp. on Discrete Algorithms, SODA’06}, 2006.

\bibitem[AHN04]{lower}
Oswin Aichholzer, Ferran Hurtado, and Marc Noy. 
A lower bound on the number of triangulations of planar point sets. 
\emph{Comput. Geom.: Theory and Applications} 29 (2004), 135--145.


\bibitem[AD+00]{bauescyclic}
Christos A. Athanasiadis, Jes\'us A. De Loera, Vic Reiner, Francisco Santos,
Fiber polytopes for the projections between cyclic polytopes,
\emph{European J. Combin.} 21:1 (2000), 19--47.


\bibitem[AS02]{AzaSan}
Miguel Azaola and Francisco Santos, The number of triangulations of the cyclic polytope $C(n,n-4)$. 
\emph{Discrete Comput. Geom.}, 27 (2002), 29--48.

\bibitem[BW10]{BeWi}
Lowell Beineke and Robin Wilson. The early history of the brick factory problem. \emph{Math. Intelligencer}, 32 (2010), 41–-48.

\bibitem[BL+92]{OMbook}
Anders Bj\"orner, Michel Las Vergnas, Bernd Sturmfels, Neil White, and G\"unter M. Ziegler. \emph{Oriented Matroids}. Cambridge University Press, Cambridge, 1992.

\bibitem[DRS10]{triangbook}
Jes\'{u}s~A. De~Loera, J\"{o}rg Rambau, and Francisco Santos.
\newblock {\em Triangulations. Structures for algorithms and applications}, 
\newblock Springer-Verlag, Berlin, 2010.

\bibitem[GKZ90]{GKZpaper}
Israel M. Gelfand, Mikhail M. Kapranov, and Andrei V. Zelevinsky. Newton polyhedra of principal $A$-determinants. 
\emph{Soviet Math. Dokl.}, 40 (1990), 278--281.

\bibitem[GKZ94]{GKZbook} 
Israel M. Gelfand, Mikhail M. Kapranov, and Andrei V. Zelevinsky. \emph{Discriminants, Resultants and Multidimensional Determinants}. Birkh\"auser, Boston, 1994.

\bibitem[Guy60]{Guy} 
Richard K. Guy. A combinatorial problem. \emph{Bull. Malayan Math. Soc.}, 7 (1960), 68--72.

\bibitem[HH63]{HaHi} 
Frank Harary and Anthony Hill, On the number of crossings in a complete graph, \emph{Proc. Edinburgh Math. Soc. (2)} 13 (1963), 333--338.

\bibitem[KVY23]{KVY21}
Andrey Kupavskii, Aleksei Volostnov, Yury Yarovikov,
Minimum number of partial triangulations.
\emph{European J.  Combin} 108 (2023), 103636.

\bibitem[Lee89]{Lee89} 
Carl W. Lee. The associahedron and triangulations of the $n$-gon. \emph{European J. Combin.} 10(6) (1989), 551--560.

\bibitem[RW23]{RW23} 
Daniel Rutschmann and Manuel Wettstein. 
Chains, Koch Chains, and Point Sets with Many Triangulations. 
\emph{J. ACM} 70(3) (2023), Article 18, 26 pages. 
DOI: \url{https://doi.org/10.1145/3585535}

\bibitem[Sch21]{Schaefer} 
Marcus Schaefer,
The Graph Crossing Number and its Variants: A Survey.
\emph{Elect. J. Combin.}, Dynamic survey DS21, last updated May 2024. DOI: \url{https://doi.org/10.37236/2713}


\bibitem[OEIS]{OEIS} 
Neil J. A. Sloane, editor, The On-Line Encyclopedia of Integer Sequences, published electronically at \url{https://oeis.org}.

\bibitem[Zie94]{Ziegler} 
G\"unter M. Ziegler. \emph{Lectures on Polytopes}. Springer-Verlag, New York, 1994.

\end{thebibliography}
\end{document}